\documentclass[12pt]{amsart}
\usepackage{amsmath}
\usepackage{amsthm}
\usepackage{amssymb}
\usepackage{tikz}
\usepackage{physics}
\usetikzlibrary{fit}
\usepackage{mathptmx}

\usepackage[T1]{fontenc}
\usepackage{enumerate}
\usepackage{hyperref}
\usepackage{etoolbox}

\usepackage{fancyhdr}
\patchcmd{\section}{\scshape}{\bfseries}{}{}

\newenvironment{nouppercase}{%
  \renewcommand{\uppercasenonmath}[1]{}}{}

\newcommand{\R}{{\mathbb{R}}}

\newcommand{\Z}{{\mathbb{Z}}}

\newtheorem{theorem}{Theorem}

\newtheorem{lemma}[theorem]{Lemma}
\newtheorem{corollary}[theorem]{Corollary}

\theoremstyle{definition}

\theoremstyle{remark}

\newcommand{\Modulus}[1]{\left| #1 \right|}

\renewcommand{\epsilon}{\varepsilon}

\begin{document}
\title[~]
{Degree bound of P\'olya Positivstellensatz}
\author[~]
{Tan Ze Kang}
\dedicatory{Mentor: Colin Tan}



\begin{nouppercase}
\maketitle
\end{nouppercase}

\begin{abstract}
  P\'olya's Positivstellensatz on the $1$-simplex says that if $P(x)$ is a real polynomial such that $P(x)>0$ whenever $x \ge 0$,
then all the coefficients of $(1+x)^mP(x)$ are positive whenever $m$ is large.
Powers-Reznick gave a complexity estimate for P\'olya's Positivstellensatz. Namely, they proved that, for such $P(x)$ of degree $d$, all the coefficients of $(1+x)^mP(x)$ are positive whenever $m > \frac{1}{2} (d^2 -d) \frac{L(P)}{\lambda(P)} - d$. where $\frac{L(P)}{\lambda(P)}$ is an invariant of $P(x)$. For $d=3$ and $d=4$ specifically, we improve Powers-Reznick's bound by showing $m > \frac{3}{2} \frac{L(P)}{\lambda(P)} - 1$ for $d=3$ and $ m > \frac{4232}{2505} \frac{L(P)}{\lambda(P)} - 1$ for $d=4$.
\end{abstract}

\section{Introduction and main result}\label{intro}
A positivstellensatz certifies the strict positivity of a polynomial $f \in \R[x] := \R[x_1, \ldots, x_n]$ on a semialgebraic set $k \subseteq R^n$ by representing $f$ as an algebraic expression. This algebraic expression of $f$ witnesses the strict positivity of $f$ on $k$.
P\'olya proved a Positivstellensatz for real homogeneous polynomials on the $n$-simplex \cite{Polya28} (reproduced in \cite[pp. 57-60]{HLP}).
For the $1$-simplex an equivalent formulation is that if $P(x) \in \R[x]$ is a polynomial such that $P(x)>0$ whenever $x \ge 0$,
then all the coefficients of $(1+x)^mP(x)$ are positive whenever $m$ is large (see also \cite[corollary 5]{Powers}).
In this case, a positivstellensatz certifies that $P(x)\in \R[x]$ on a nonnegative real line by representing $P(x)$ as an algebraic expression : $(1+x)^mP(x)$ has positive coefficient.

Powers-Reznick obtained an upper bound for the least $m$ such that all the coefficients of $(1+x)^mP(x)$ are positive \cite{Reznick}. For a polynomial $P(x) = \sum_{j = 0}^d a_j x^j \in \R[x]$ of degree $d$ such that $P(x) > 0$ whenever $x \ge 0$, they showed that all the coefficients of $(1+x)^mP(x)$ are positive whenever 
\begin{equation} \label{reznick bound}
    m  > \frac{1}{2}(d^2 - d) \frac{L(P)}{\lambda(P)} - d,
\end{equation}
where
\begin{equation} \label{eq: LPAndLambdaP}
    L(P) := \max_{j = 0,\ldots, d} \frac{1}{\binom{d}{j}}|a_j| 
\quad \text{and} \quad
    \lambda(P) := \inf_{x \in [0, \infty)} \frac{P(x)}{(1+x)^d},
\end{equation}
We remark that $\lambda(P)$ is a well-defined positive real number for such $P(x)$ since $P(x)$ is a polynomial of degree $d$ with positive leading coefficient so that $\lim_{x \to \infty} \frac{P(x)}{(1+x)^d}$ is a positive real number. 
Next, we shall show 
\begin{equation} \label{inequality}
L(P)\ge\lambda(P).
\end{equation}
Note that 
\[
L(P)\ge \max_{j = 0,\ldots, d} \frac{1}{\binom{d}{j}}a_j  
\quad \text{and} \quad 
\sup_{x \in [0, \infty)} \frac{\sum_{j=0}^d a_jx^j}{(1+x)^d} \ge \lambda(P)
\]
Hence it suffices to show that $\max_{j = 0,\ldots, d} \frac{1}{\binom{d}{j}}a_j \ge \sup_{x \in [0, \infty)}\frac{\sum_{j=0}^d a_jx^j}{(1+x)^d}$.
Let $A = \max_{j = 0,\ldots, d} \frac{1}{\binom{d}{j}}a_j $, we know that $A\ge \frac{1}{\binom{d}{j}}a_j$.
$\frac{\sum_{j=0}^d a_jx^j}{(1+x)^d} \le \frac{\sum_{j=0}^d \binom{d}{j}Ax^j}{(1+x)^d} = \frac{(1+x)^d}{(1+x)^d} A$.
Hence $\frac{\sum_{j=0}^d a_jx^j}{(1+x)^d} \le  \max_{j = 0,\ldots, d} \frac{1}{\binom{d}{j}}a_j$. 
This proves \eqref{inequality} and the equality holds when $ \sup_{x \in [0, \infty)}\frac{\sum_{j=0}^d a_jx^j}{(1+x)^d} = \lambda(P)$, which occurs only when $P(x)= b \cdot (1 + x)^d$ for some $b$.

In this report we improve Powers-Reznick's degree bound for $d =3$ and $d=4$.            
\begin{theorem}\label{mainthm}
Let $P(x) \in \R[x]$ be a polynomial of degree $d$ such that $P(x)>0$ whenever $x \ge 0$.
For $d=1,\ldots, 4$, all the coefficients of $(1+x)^mP(x)$ are positive whenever
\begin{equation} \label{our bound}
m > C_d \frac{L(P)}{ \lambda(P)} - 1,
\end{equation}
where
$C_d$ is given in the table below
\end{theorem}


\begin{table}[h!] \label{improved}
\begin{center}
\begin{tabular}{ c | c c c c }
 $d$ & $1$ & $2$ & $3$ & $4$  \\ 
 \hline
 $C_d$ & $0$ & $1$ & $\frac{3}{2}$ & $\frac{4232}{2505}$  \\  
 $\frac{d^2-d}{2}$ & $0$ & $1$ & $3$ & $6$  \\ 
\end{tabular}
\caption{Improvement in degree bounds}
\end{center}
\end{table}

Comparing \eqref{reznick bound} and \eqref{our bound}, when $C_d\frac{L(P)}{\lambda(P)} -1 \le \frac{1}{2}(d^2 -d)\frac{L(P)}{\lambda(P)} -d $ holds, we improved Powers-Reznick's bound. 
For $d=3$, the above inequality holds whenever $\frac{L(P)}{\lambda(P)} \ge \frac{4}{3}$. For $d=4$, the inequality holds by \eqref{inequality} for all $P(x)$ such that $P(x)>0$ whenever $x\ge 0$ .

To prove that $(1+x)^mP(x)$ have positive coefficients, we use the equivalent condition:
\begin{equation} \label{eq: equivalent criterion}
[x^{c(m+d)}](1+x)^mP(x)>0 \text{ whenever } c(m+d) \text{ is a non-negative integer for } 0\le c\le 1.
\end{equation}
P\'olya chose $\frac{m!(m+d)^d}{(c(m+d))!}(1-c)^dP(\frac{c}{1-c})$ to approximate $[x^{c(m+d)}](1+x)^mP(x)$. When this approximator is used, the error $[x^{c(m+d)}](1+x)^mP(x) - \frac{m!(m+d)^d}{(c(m+d))!}(1-c)^dP(\frac{c}{1-c})$ has a sign independent of $P(x)$.
We choose our approximator to be $\binom{m+d}{c(m+d)}(1-c)^dP(\frac{c}{1-c})$ instead such that the error is $0$ when $P(x) = (1+x)^d$.

\section{Proof of theorem \ref{mainthm}}
\begin{lemma}\label{lem:identity}
Let $P(x) = \sum_{j = 0}^d a_j x^j\in \R[x]$ be a polynomial of degree $d$. 
For each $m \in \Z_{\ge 0}$ and each $0 \le c \le 1$ such that $c(m + d) \in \Z_{\ge 0}$,
 \begin{align*}
     &[x^{c(m+d)}](1+x)^mP(x) - \binom{m+d}{c(m+d)}(1-c)^dP(\frac{c}{1-c})\\
        = {} & \binom{m+d}{c(m+d)}\sum_{j=0}^d a_j (f_{c}^{(j)}(\frac{1}{m+d})-f_{c}^{(j)}(0)),
 \end{align*}
where 
\begin{equation} \label{eq: fcx}
f_{c}^{(j)}(x):= \frac{\left(c)(c - x) \cdots (c - (j - 1)x \right)
                \left(1 - c)(1 - c - x) \cdots (1 - c - (d - j - 1)x\right)}
                    {(1)(1 - x)(1 - 2x) \cdots (1 - (d- 1)x)}.
\end{equation}

\end{lemma}
\begin{proof}

Since $[x^{c(m+d)}](1+x)^mP(x) = \sum_{j=0}^d a_j\binom{m}{c(m+d)-j}$ and
$
    \binom{m}{c(m+d)-j} 
     =\binom{m+d}{c(m+d)}f_c^{(j)}(\frac{1}{m + d}),
$ for $0\le j\le d$, 
hence
\begin{equation}\label{coefficient of (1+x)^mPx}
[x^{c(m+d)}](1+x)^mP(x) = \binom{m+d}{c(m+d)}\sum_{j=0}^d a_j f_c^{(j)}(\frac{1}{m+d}).
\end{equation}
Combining \eqref{coefficient of (1+x)^mPx} with $(1-c)^dP(\frac{c}{1-c})=\sum_{j=0}^d a_j c^j(1-c)^{d-j}=f_c^{(j)}(0)$, we will prove lemma \ref{lem:identity}. 
\end{proof}

From \eqref{eq: fcx}, 
\begin{equation}\label{fcj3}
f_c^{(j)}(\frac{1}{m+d})-f_c^{(j)}(0)= \frac{c(1-c)}{(m+1)(m+2)\cdots (m + d - 1)}h_c^{(j)}(m),  
\end{equation}
for some polynomial $h_c^{(j)}(m)$ of degree $d-2$ in $m$ and of degree $d - 2$ in $c$.
 From \eqref{eq: fcx}, by substituting $c$ with $1-c$, and $d$ wtih $d-j$, we can obtain $f_c^{(j)}(x) = f_{1-c}^{(d-j)}(x)$. Hence, from \eqref{fcj3}, we can show that
\begin{equation} \label{eq: symmetry}
    h_c^{(j)}(m) = h_{1-c}^{d - j}(m)
\end{equation}

\begin{lemma}\label{limit}
For $d=3$ or $4$, let $f_c^{(j)}(m)$ be as in \eqref{eq: fcx}. 
For each $m\ge 1$ and $0 \le c \le \frac{1}{2}$,
\[
\sum_{j = 0}^d \binom{d}{j}|f_c^{(j)}(\frac{1}{m+d})-f_c^{(j)}(0)|\le \frac{C_d}{m + 1},
\]
where $C_3=\frac{3}{2}$ and $C_4 = \frac{4232}{2505}$ as given in table \ref{improved}.
\end{lemma}

\begin{proof}[Proof of lemma \ref{limit} for $d=3$]
We will first determine the sign of $f_c^{(j)}(\frac{1}{m+3})-f_c^{(j)}(0)$. For $d = 3$, we calculate $h_c^{(j)}(m)$ explicitly as follows:
\begin{equation}\label{hcjm}
h_c^{(j)}(m)= 
    \begin{cases}
    (-3+3c)m+(-5+7c) & \text{if } j=0 \\
    (2-3c)m+(4-7c) & \text{if } j=1 \\
    \end{cases}
\end{equation}
Note that $h_c^{(2)}(m) = h_{1-c}^{(1)}(m)$ and $h_c^{(3)}(m) = h_{1-c}^{(0)}(m)$ from \eqref{eq: symmetry}.
Since $h_c^{(j)}(m)$ is linear in $m$, we can determine the sign as $0\le c\le1$.

\begin{table}[h!]
\begin{center}
\begin{tabular}{ c | c c c c c c c }
 $j$ & $0\le c\le \frac{2}{7}$ & $\frac{2}{7}\le c
 \le \frac{1}{3}$ & $\frac{1}{3}\le c
 \le \frac{3}{7}$ & $\frac{3}{7}\le c
 \le \frac{4}{7}$ & $\frac{4}{7}\le c
 \le \frac{2}{3}$ & $\frac{2}{3}\le c
 \le \frac{5}{7}$ & $\frac{5}{7}\le c
 \le 1$ \\
 \hline
 0 & $\le0$ & $\le0$ & $\le0$ & $\le0$ & $\le0$ & $\le0$ & ? \\  
 1 & $\ge0$ & $\ge0$ & $\ge0$ & $\ge0$ & ? & $\le0$ & $\le0$ \\ 
 2 & $\le0$ & $\le0$ & ? & $\ge0$ & $\ge0$ & $\ge0$ & $\ge0$ \\  
 3 & ? & $\le0$ & $\le0$ & $\le0$ & $\le0$ & $\le0$ & $\le0$ \\ 
\end{tabular}
\caption{The entry of the $j$-th row and the column $a\le c\le b$ is the sign of $h_c^{(j)}(m)$ for $a\le c\le b$ and $m\ge 0$. When there is a ?, it means that the sign depends on $m$. }
\end{center}
\end{table}

For example,  $j=1$ when $0\le c\le \frac{4}{7}$, the linear function $h_c^{(1)}(m)$ has positive leading coefficient since $2-3c\ge 2-3(\frac{4}{7})=\frac{2}{7}>0$ and hence, 
\[
h_c^{(1)}(m)\ge h_c^{(1)}(0)=4-7c\ge0, \quad \text{for } 0\le c\le \frac{4}{7}
\]

The possible signs of $(h_c^{(0)}(m),h_c^{(1)}(m),h_c^{(2)}(m),h_c^{(3)}(m))$ are given by \[
(h_c^{(0)}(m),h_c^{(1)}(m),h_c^{(2)}(m),h_c^{(3)}(m))
\begin{cases}
(\le0,\ge0,\le0,\ge0) & \text{only if } 0\le c\le\frac{2}{7}\\
(\le0,\ge0,\le0,\le0) & \text{only if } 0\le c\le\frac{3}{7}\\
(\le0,\ge0,\ge0,\le0) & \text{only if } \frac{1}{3}\le c\le\frac{1}{2}\\
\end{cases}
\]
We return to the estimation of the sum $\sum_{j = 0}^3 \binom{3}{j}|f_c^{(j)}(\frac{1}{m+3})-f_c^{(j)}(0)|$.

\paragraph{\textbf{Case 1}: $(h_c^{(0)}(m),h_c^{(1)}(m),h_c^{(2)}(m),h_c^{(3)}(m)) = (\le0,\ge0,\le0,\ge0)$ and $0\le c\le\frac{2}{7}$}
Note that for this case, using \eqref{hcjm} with its respective signs, we will obtain
$
\sum_{j = 0}^3 \binom{3}{j}|h_c^{(j)}(m)|=
(12-24c)m+(28-56c)
$.
Hence, using \eqref{fcj3}, we will get
\begin{equation} \label{7/3}
\begin{aligned} [b]
\sum_{j = 0}^3\binom{3}{j}|f_c^{(j)}(m)-f_c^{(j)}(0)|
=(12+\frac{4}{m+2})\frac{c(1-c)(1-2c)}{m+1}
\le \frac{7}{3\sqrt{3}}\frac{1}{m+1},
\end{aligned}
\end{equation}
where the last inequality follows from $0<\frac{4}{m+2}\le 2$ and $c(1-c)(1-2c)\le \frac{1}{6\sqrt{3}}$ 
since $0\le c\le \frac{2}{7}$ and $m\ge0$.

\paragraph{\textbf{Case 2}: $(h_c^{(0)}(m),h_c^{(1)}(m),h_c^{(2)}(m),h_c^{(3)}(m)) = (\le0,\ge0,\le0,\le0)$ and $0\le c\le\frac{3}{7}$}
Similarly, for this case, using \eqref{hcjm} with its respective signs, we will obtain
$
\sum_{j = 0}^3 \binom{3}{j}|h_c^{(j)}(m)|=
(6-18c)m+(24-42c)
$.
Hence, using \eqref{fcj3}, we will get
\begin{equation} \label{8/81}
\begin{aligned}
\sum_{j = 0}^3\binom{3}{j}|f_c^{(j)}(m)-f_c^{(j)}(0)|
& =6(1+\frac{2-c}{m+2})\frac{c(1-c)(1-3c)}{m+1}
 \le \frac{8(7\sqrt{7}-10)}{81}\frac{1}{m+1},
\end{aligned}
\end{equation}
where the last inequality follows from $0<\frac{2-c}{m+2}\le 1$ and $c(1-c)(1-3c)\le \frac{2(7\sqrt{7}-10)}{243} $ since $0\le c\le \frac{3}{7}$ and $m\ge0$.

\paragraph{\textbf{Case 3}: $(h_c^{(0)}(m),h_c^{(1)}(m),h_c^{(2)}(m),h_c^{(3)}(m)) = (\le0,\ge0,\ge0,\le0)$ and $\frac{1}{3}\le c\le\frac{1}{2}$}
Also for this case, using \eqref{hcjm} with its respective signs, we will obtain
$
\sum_{j = 0}^3 \binom{3}{j}|h_c^{(j)}(m)|=
6m+6
$.
Hence, using \eqref{fcj3}, we will get
\begin{equation} \label{3/2}
\begin{aligned}[b]
\sum_{j = 0}^3\binom{3}{j}|f_c^{(j)}(m)-f_c^{(j)}(0)|
& =\frac{c(1-c)}{(m+1)(m+2)}\left(6m+6\right)
& \le \frac{3}{2} \frac{1}{m+1}
\end{aligned}
\end{equation}
where the last inequality follows from $c(1-c)\le \frac{1}{4}$ since
$\frac{1}{3}\le c\le \frac{1}{2}$ and $m\ge0$.

Hence, comparing \eqref{7/3}, \eqref{8/81}, \eqref{3/2}, we can conclude that 
\[
\sum_{j = 0}^3\binom{3}{j}|f_c^{(j)}(m)-f_c^{(j)}(0)| \le \max\left\{\frac{7}{3\sqrt{3}},\frac{8(7\sqrt{7}-10)}{81},\frac{3}{2}\right\}\frac{1}{m+1}
=\frac{3}{2} \frac{1}{m+1}
\]
for $0\le c\le \frac{1}{2}$.
\end{proof}

\begin{proof}[Proof of lemma \ref{limit} for $d = 4$]
We will first determine the sign of $f_c^{(j)}(\frac{1}{m+4})-f_c^{(j)}(0)$. For $d = 4$, we calculate $h_c^{(j)}(m)$ explicitly as follows:
\begin{equation}\label{icjm}
h_c^{(j)}(m)= 
    \begin{cases}
    m^2(-6c^2+12c-6)+m(-37c^2+63c-26)+(-58c^2+78c-26) & \text{if } j=0 \\
    m^2(6c^2-9c+3)+m(37c^2-50c+15)+(58c^2-68c+18) & \text{if } j=1 \\
    m^2(-6c^2+6c-1)+m(-37c^2+37c-7)+(-58c^2+58c-12) & \text{if } j=2 \\
    \end{cases}
\end{equation}
Note that $h_c^{(3)}(m) = h_{1-c}^{(1)}(m)$ and $h_c^{(4)}(m) = h_{1-c}^{(0)}(m)$ from \eqref{eq: symmetry}.

Since $h_c^{(j)}(m)$ is quadratic in $m$, we can find the turning point of the quadratic given by $m_0 =-\frac{b}{2a}$ for any quadratic $am^2+bm+c$. From \eqref{eq: symmetry}, it is sufficient to determine the sign of $h_c^{(j)}(m)$ for $j=1,2,3$.
 \[ 
 m_0 =
  \begin{cases}
    \frac{26-37c}{12(c-1)} & \text{if } j=0 \\
    -\frac{37c^2-50c+15}{12c^2-18c+6} & \text{if } j=1 \\
    -\frac{37c^2-37c+7}{12c^2-12c+2} & \text{if } j=2 \\
    \end{cases}
\]

Using the turning point above we can obtain,
\[
h_c^{(j)}(m_0)=
\begin{cases}
\frac{1}{24} (-23 c^2 - 52 c + 52) & \text{if } j=0 \\
-\frac{-23 c^4 + 20 c^3 + 34 c^2 - 36 c + 9}{24 c^2 - 36 c + 12}  & \text{if } j=1 \\
\frac{-23 c^4 + 46 c^3 - 25 c^2 + 2 c + 1}{24 c^2 - 24 c + 4} & \text{if } j=2\\
\end{cases}
\]

We can solve the inequality $m_0(c) \ge 1$ for $0\le c<\frac{1}{2}$ as follows:
\[ m_0(c)\ge 1 \quad \text{if and only if}
  \begin{cases}
    \frac{38}{49}\le c\le 1 & \text{if } j=0 \\
    \frac{34}{49} - \frac{\sqrt{127}}{49} \le c < \frac{1}{2} & \text{if } j=1 \\
    \frac{3 - \sqrt(3)}{6}<c \le \frac{7 - \sqrt(13)}{14} & \text{if } j=2 \\
    \end{cases}
\]
For the case $j=0$, 
the leading coefficient of $h_c^{(0)}(m)$ is negative for $0\le c\le \frac{1}{2}$.

\textbf{Case A}: $0 \le c \le \frac{1}{2}$. Since the leading coefficient is negative, hence $h_c^{(0)}(m)$ is is decreasing in $m$ for $m\ge m_0$. In this case $m_0\le 1$, hence $h_c^{(0)}(m)\ge h_c^{(0)}(1)$ where $h_c^{(0)}(1)$ can be solved to be negative for $0\le c\le \frac{1}{2}$.

Next, for $j=1$, 
the leading coefficient of $h_c^{(1)}(m)$ is positive for $0\le c\le \frac{1}{2}$. 

\textbf{Case $B_1$}: 
$0\le c\le \frac{127-\sqrt{1585}}{212}$. Since the leading coefficient is positive, hence $h_c^{(1)}(m)$ is increasing in m for $m\ge m_0$. In this case,  $m_0\le 1$, hence for all $m\ge 1$,
\[
h_c^{(1)}(1) \le h_c^{(1)}(m) = 101 c^2 - 127 c + 36 \ge 0,
\]
where the last inequality holds since $ c\le \frac{127-\sqrt{1585}}{212}$.

\textbf{Case $B_2$}:
$\frac{34-\sqrt{127}}{49}\le c< \frac{1}{2} $. Since the leading coefficient is positive, and in this case,  $m_0\ge 1$, hence for all $m\ge 1$,
$h_c^{(1)}(m_0) \le h_c^{(1)}(m)$ where  $h_c^{(1)}(m_0) \le -\frac{3(215\sqrt{127}-1497)}{2401} < 0$ for $\frac{34-\sqrt{127}}{49}\le c< \frac{1}{2} $.

Lastly, for $j=2$, 
the leading coefficient of $h_c^{(2)}(m)$ is negative for $0\le c\le \frac{1}{6}(3-\sqrt{3})$ and positive for $\frac{1}{6}(3-\sqrt{3})\le c\le \frac{1}{2}$. 

\textbf{Case $C_1$}: 
$0\le c< \frac{1}{6}(3-\sqrt{3})$. Since the leading coefficient is negative, hence $h_c^{(2)}(m)$ is decreasing in $m$ for $m\ge m_0$. In this case,  $m_0\le 1$, hence for all $m\ge 1$, 
\[
h_c^{(2)}(m) \le h_c^{(2)}(1) =  -101 c^2 + 101 c - 20 \le -\frac{19}{6} < 0,
\] where the second last inequality holds since $ c\le \frac{1}{6}(3-\sqrt{3})$.

\textbf{Case $C_2$}:
$\frac{1}{6}(3-\sqrt{3})< c\le \frac{1}{2}-\frac{\sqrt{13}}{14}$. Since the leading coefficient is positive, hence $i_c^{(2)}(m)$ is increasing in m for $m\ge m_0$. In this case,  $m_0\ge 1$, hence for all $m\ge 1$,
$h_c^{(2)}(1) \le h_c^{(2)}(m)$. For $\frac{1}{6}(3-\sqrt{3})\le c\le \frac{7-\sqrt{13}}{14}$, $ h_c^{(2)}(1) \le - \frac{71}{49} < 0$.

\textbf{Case $C_3$}:
$\frac{7-\sqrt{13}}{14}\le c < \frac{1}{2}$. Since the leading coefficient is positive and in this case,  $m_0\le 1$, hence $h_c^{(2)}(m_0) \le h_c^{(2)}(m)$ where $h_c^{(2)}(m_0)$ is solved to be $ h_c^{(2)}(m_0)\le -\frac{1}{32} < 0$ for $\frac{7-\sqrt{13}}{14}\le c < \frac{1}{2}$.

\begin{table}[h!]
\begin{center}
\begin{tabular}{ c | c c c c c c c}
 $j$ &
 $0\le c< \frac{3-\sqrt{3}}{6}$ & 
 $\frac{3-\sqrt{3}}{6}< c\le \frac{127-\sqrt{1585}}{202}$ & 
 $\frac{127-\sqrt{1585}}{202}\le c< \frac{1}{2}$ & 
 $c=\frac{1}{2}$\\
 \hline
 0 & $\le0$ (A) & $\le0$ (A) & $\le0$ (A) & $\le0$ (A)\\  
 1 & $\ge0$ ($B_1$) & $\ge0$ ($B_1$)& ? & $\le0$ ($B_2$)\\ 
 2 & $\le0$ ($C_1$) & ? & ? & $\ge0$ \\ 
\end{tabular}
\caption{The entry of the $j$-th row and the column(range of $c$) is the sign of $h_c^{(j)}(m)$ for that range. When there is a ?, it means that the sign depends on $m$. }
\end{center}
\end{table}

Hence the possible signs of $(h_c^{(0)}(m),h_c^{(1)}(m),h_c^{(2)}(m),h_c^{(3)}(m),h_c^{(4)}(m))$ are given by \[
\begin{cases}
(\le0,\ge0,\le0,\ge0,\le0) & \text{only if } 0\le c<\frac{1}{2}\\
(\le0,\ge0,\ge0,\ge0,\le0) & \text{only if } \frac{3-\sqrt{3}}{6}\le c<\frac{1}{2}\\
(\le0,\le0,\le0,\le0,\le0) & \text{only if } \frac{127-\sqrt{1585}}{202}\le c<\frac{1}{2}\\
(\le0,\le0,\ge0,\le0,\le0) & \text{only if } \frac{127-\sqrt{1585}}{202}\le c\le\frac{1}{2}\\
\end{cases}
\]

\begin{equation}\label{psi}
\sum_{j = 0}^4 \binom{4}{j}|f_c^{(j)}(m)-f_c^{(j)}(0)|
= \frac{c(1-c)(\phi(m) c^2 - \phi(m) c + \psi(m)) }{(m+1)(m+2)(m+3)}
 \le \frac{( \psi(m) )^2 }{4\phi(m)(m+1)(m+2)(m+3)} ,
\end{equation}
Let $g(c):= c(1-c)(\phi(m) c^2 - \phi(m) c + \psi(m)) $, where $\phi(m) \psi(m)\ge 0$ and $\phi(m)(\phi(m) - 2\psi(m)) \ge 0$. The above inequality follows since $g(c)\le g(\frac{\phi(m)-\sqrt{\phi(m)(\phi(m) -2 \psi(m))}}{2\phi(m)}) = \frac{\psi(m)^2}{4(\phi(m)}$.

\textbf{Case 1}: $(\le0,\ge0,\le0,\ge0,\le0)$. 
In this case, $\phi(m)= 96m^2 +592m + 982 $ and $\psi(m) = 24m^2 +136m +208$.
From \eqref{psi},
\begin{equation}\label{4232/2505}
\sum_{j = 0}^4 \binom{4}{j}|f_c^{(j)}(m)-f_c^{(j)}(0)|
= \frac{8 (3 m^2 + 17 m + 26)^2}{((m+2)(m + 3) (48 m^2 + 296 m + 491)} \cdot \frac{1}{m+1}
\le\frac{4232}{2505}\frac{1}{m+1}.
\end{equation}
This inequality holds since $\frac{8 (3 m^2 + 17 m + 26)^2}{((m+2)(m + 3) (48 m^2 + 296 m + 491)}$ is decreasing in $m$ for $m\ge 1$ so that $\frac{8 (3 m^2 + 17 m + 26)^2}{((m+2)(m + 3) (48 m^2 + 296 m + 491)}\le \frac{8 (3 \cdot 1^2 + 17 \cdot 1 + 26)^2}{((1+2)(1 + 3) (48 \cdot 1^2 + 296 \cdot 1 + 491)}=\frac{4232}{2505}$.

\textbf{Case 2}: $(\le0,\ge0,\ge0,\ge0,\le0)$.
In this case, $\phi(m)= 24m^2 +128m + 232 $ and $\psi(m) = 12m^2 +52m +64$.
From \eqref{psi},
\begin{equation}\label{8/9}
\begin{aligned} 
&\sum_{j = 0}^4 \binom{4}{j}|f_c^{(j)}(m)-f_c^{(j)}(0)|
& \le \frac{8}{9} \frac{1}{m+1}.
\end{aligned}
\end{equation}

\textbf{Case 3}: $(\le0,\le0,\le0,\le0,\le0)$ 
In this case, 
\begin{equation}\label{0}
    \sum_{j = 0}^4 \binom{4}{j}|f_c^{(j)}(m)-f_c^{(j)}(0)|= 0.
\end{equation}

\textbf{Case 4}: $(\le0,\le0,\ge0,\le0,\le0)$.
In this case,  
\begin{equation}
\begin{aligned} \label{3/2'}
&\sum_{j = 0}^4 \binom{4}{j}|f_c^{(j)}(m)-f_c^{(j)}(0)|\\
& = \frac{c(1-c)(c^2(-72m^2-444m-696)+c(72m^2+444m+696)-(12m^2+84m+144))}{(m+1)(m+2)(m+3)}\\
& =  \frac{12c(1-c)(-c^2(6+\frac{7}{m+2}+\frac{1}{(m+2)(m+3)})+c(6+\frac{7}{m+2}+\frac{1}{(m+2)(m+3)})-(1+\frac{2}{m+2}))}{m+1}\\
& < \frac{3}{2} \frac{1}{m+1},
\end{aligned}
\end{equation}
where the last inequality holds since  $(-c^2(6+\frac{7}{m+2}+\frac{1}{(m+2)(m+3)})+c(6+\frac{7}{m+2}+\frac{1}{(m+2)(m+3)})-(1+\frac{2}{m+2}))\le \frac{2m+5}{4m+12} < \frac{1}{2}$  and $c(1-c)\le \frac{1}{4}$ at $c=\frac{1}{2}$, the inequality is true.

Hence, comparing \eqref{4232/2505}, \eqref{8/9}, \eqref{0},\eqref{3/2'}, we can conclude that 
\[
\sum_{j = 0}^4\binom{4}{j}|f_c^{(j)}(m)-f_c^{(j)}(0)| \le \max\left\{\frac{4232}{2505},\frac{8}{9},0,\frac{3}{2}\right\}\frac{1}{m+1}
=\frac{4232}{2505} \frac{1}{m+1}
\]
for $0\le c\le \frac{1}{2}$.
\end{proof}

\begin{lemma}\label{1-c}
For $0\le c\le 1$ and each $d\in \Z_{\ge0}$
\[\sum_{j=0}^d \binom{d}{j}|f_c^{(j)}(\frac{1}{m+d})-f_c^{(j)}(0)| = 
\sum_{j=0}^d \binom{d}{j}|f_{1-c}^{(j)}(\frac{1}{m+d})-f_{1-c}^{(j)}(0)|
\]
\end{lemma}

\begin{proof}
For $j = 0,\ldots, d$, by \eqref{eq: symmetry}, 
we have $
    f_c^{(j)}(\frac{1}{m+d})-f_c^{(j)}(0) = f_{1-c}^{(d -j)}(\frac{1}{m+d})-f_{1-c}^{(d-j)}(0)
$.
Hence
\begin{align*}
    \sum_{j=0}^d \binom{d}{j}|f_c^{(j)}(\frac{1}{m+d})-f_c^{(j)}(0)|
    &= \sum_{j=0}^d \binom{d}{j}|f_{1-c}^{(d -j)}(\frac{1}{m+d})-f_{1-c}^{(d-j)}(0)| \\
         &= \sum_{j=0}^d \binom{d}{j}|f_{1-c}^{(j)}(\frac{1}{m+d})-f_{1-c}^{(j)}(0)|,
\end{align*}
where the last equality follows by replacing the index $j$ with $d - j$ and using the identity $\binom{d}{d-j} = \binom{d}{j}$. 
\end{proof}

\begin{corollary}\label{c>1/2}
For $d=3$ or $4$, let $f_c^{(j)}(m)$ be as in \eqref{eq: fcx}. 
For each $m \ge 1$ and $0 \le c \le 1$,
\[
\sum_{j = 0}^d \binom{d}{j}|f_c^{(j)}(\frac{1}{m+d})-f_c^{(j)}(0)|\le \frac{C_d}{m + 1},
\]
where $C_3=\frac{3}{2}$ and $C_4 = \frac{4232}{2505}$ as given in table \ref{improved}.
\end{corollary}
\begin{proof}
From lemma \ref{limit}, for $c\le\frac{1}{2}$ we have $\sum_{j = 0}^3\binom{3}{j}|f_c^{(j)}(m)-f_c^{(j)}(0)| \le \frac{3}{2} \frac{1}{m+1}$  Using lemma \ref{1-c}, we can show that for $1-c\le \frac{1}{2}$, $\sum_{j=0}^d \binom{d}{j}|f_c^{(j)}(\frac{1}{m+d})-f_c^{(j)}(0)| =\sum_{j = 0}^3\binom{3}{j}|f_{1-c}^{(j)}(m)-f_{1-c}^{(j)}(0)| \le \frac{3}{2} \frac{1}{m+1}$, where the last inequality follows from lemma \ref{limit} (with $c$ replaced by $1-c$).
\end{proof}

\begin{proof}[Proof of theorem \ref{mainthm}]
Let $P(x)\in \R[x]$ be a polynomial of degree d. We may assume $d=3$ or $4$ since Theorem \ref{mainthm} follows Power-Reznick's bound for $d=1$ and $2$ from \eqref{reznick bound}. By lemma \ref{lem:identity} and the definition of $L(P)$ in \eqref{eq: LPAndLambdaP},
\begin{align*}
   & |[x^{c(m+d)}](1+x)^mP(x) - \binom{m+d}{c(m+d)}(1-c)^dP(\frac{c}{1-c})| \\
        &\le \binom{m+d}{c(m+d)} L(P)\sum_{j = 0}^d \binom{d}{j}|f_c^{(j)}(\frac{1}{m+d})-f_c^{(j)}(0)| \\
        &\le \binom{m+d}{c(m+d)} L(P)\frac{C_d}{m+1},
\end{align*}
where the last line follows from corollary \ref{c>1/2} for $m\ge 1$.
Hence, 
\begin{align*}
[x^{c(m+d)}](1+x)^mP(x) &\ge \binom{m+d}{c(m+d)}\left((1-c)^dP(\frac{c}{1-c}) - L(P)\frac{C_d}{m+1}\right) \\
&\ge \binom{m+d}{c(m+d)}\left(\lambda(P) - L(P)\frac{C_d}{m+1}\right),
\end{align*}
where the last inequality is true since  $ \lambda(P) \le (1-c)^dP(\frac{c}{1-c}) $ by taking $x=\frac{c}{1-c}$ in the definition of $\lambda(P)$ in \eqref{eq: LPAndLambdaP}.
Thus $[x^{c(m+d)}](1+x)^mP(x)$ is positive whenever $m > C_d \frac{L(P)}{ \lambda(P)} - 1$ and $m \ge 1$. But $m > C_d \frac{L(P)}{\lambda(P)} -1 \ge \frac{1}{2}$ by $C_d \ge \frac{3}{2}$ and $\frac{L(P)}{\lambda(P)} \ge 1$ from \eqref{inequality} and hence $m \ge 1$ since $m$ is an integer. So $m > C_d \frac{L(P)}{ \lambda(P)} - 1$ is the only condition required. Therefore, by the equivalent condition given in \eqref{eq: equivalent criterion}, all the coefficients of $(1+x)^mP(x)$ are positive for such $m$.
\end{proof}

\section{Future work}

We will run a different method to prove lemma \ref{limit} for general $d$. Let $f_c^{(j)}(m)$ be as in \eqref{eq: fcx}. 
For each positive integer $d$, we wish to find a constant $C_d >0 $ such that, for each $m \ge 1$ and $0 \le c \le \frac{1}{2}$,
\begin{equation} \label{lem: essentialLemma}
\sum_{j = 0}^d \binom{d}{j}|f_c^{(j)}(\frac{1}{m+d})-f_c^{(j)}(0)|\le \frac{C_d}{m + 1},
\end{equation}

\begin{proof}
Since $f_c^{(j)}(x)$ is a rational function whose numerator is a polynomial of degree at most $(d - 1)$ and whose denominator is a polynomial of degree $d - 1$, by the theory of partial fractions,
\begin{equation} \label{eq: partialFractionDecomposition}
f_c^{(j)}(x) = \gamma_c^{(j)} + \sum_{r = 1}^{d - 1} \frac{\alpha_c^{(j)}(r)}{1 - rx},
\end{equation}
where
\begin{equation} \label{eq: defineAlphacjr}
    \alpha_c^{(j)}(r) = \frac{(-1)^{d - r - 1} d}{r^2 \binom{d}{j}} \binom{cr}{j} \binom{(1 - c) r}{d - j} \binom{d - 1}{r}.
        \quad {\text{ for }} r = 1 ,\ldots, d - 1.
\end{equation}
The value of the constant term $\gamma_c^{(j)}$ will not concern us here.
Since $f_c^{(j)}(x)-f_c^{(j)}(0)=\sum_{r=1}^{d-1}\alpha_c^{(j)}(r)(\frac{1}{1-rx}-1) = x\sum_{r=1}^{d-1}\alpha_c^{(j)}(r)(\frac{r}{1-rx})$, hence for $x \ge 0$,
\[
|f_c^{(j)}(x)-f_c^{(j)}(0)|\le x\sum_{r=1}^{d-1}|\alpha_c^{(j)}(r)|(\frac{r}{1-rx})
\]
Hence
\[
\sum_{j = 0}^d \binom{d}{j}|f_c^{(j)}(x)-f_c^{(j)}(0)|\le 
 c(1-c)x\sum_{r=1}^{d-1}\Modulus{\frac{r}{1 - rx}} Q_c(r)
    \quad \text{where }Q_c(r):=\sum_{j=0}^d \binom{d}{j} \frac{|\alpha_c^{(j)}(r)|}{c(1-c)}.
\]
Since $Q_c(r)\ge0$ and $\frac{r}{1-rx}$ is increasing in $x$ for $x \le \frac{1}{d} < \frac{1}{r}$, for $r=1,\ldots, d-1$, hence, for $0\le x\le \frac{1}{d}$,
\begin{equation}\label{bound}
\sum_{j = 0}^d \binom{d}{j}|f_c^{(j)}(x)-f_c^{(j)}(0)|\le c(1-c)x\sum_{r=1}^{d-1}\frac{dr}{d-r}Q_c(r).
\end{equation}

For $d = 5$, it can be shown numerically that, for $0\le c \le \frac{1}{2}$,
\[
 c(1-c)\sum_{r=1}^{5-1}\frac{5r}{5-r}Q_c(r) < 16.5
\]
Hence, together with \eqref{bound}, we can take $C_5 = 16.5$ in \eqref{lem: essentialLemma}.
\end{proof}



Degree bounds for P\'olya's positvstellensatz for both low degree ,$d$, and general $d$ have various applications to different fields. We have chose to work on small $d$ for its use in optimization control theory. One example is that $d=2$ is used to find the stability number of a graph. On the other hand, general $d$ have various applications to the complexity of archimedean positivstellensatze. For example, Schweighofer showed that a degree bound for P\'olya's positvstellensatz implies a corrosponding degree bound for Schm\"udgen's positivstellensatz \cite{Schweighofer}, 
and Nie-Schweighofer showed that a degree bound for P\'olya's positvstellensatz implies a corrosponding degree bound for Putinar's positivstellensatz \cite{Putinar}.
There are also applications of degree bounds for P\'olya positvstellensatz to estimate the rate of convergence of hierarchy of lower bounds that converge to infimum of fixed degree polynomials on the simplex \cite{Ahmadi}.

\paragraph{\textbf{Acknowledgements}}
Thank you Dr Colin Tan
for guiding me through this project. Special thanks to 
Mr Lim Teck Choow,
Professor Wing-Keung To and 
Professor CheeWhye Chin for their help.
An earlier version of this report was submitted to Singapore Mathematical Medley under the title \lq\lq Complexity of P\'olya Positivstellensatz for polynomial of low degree \rq\rq.

\end{document}